\documentclass[11pt]{amsart}
\usepackage{latexsym}
\usepackage{amsxtra}
\usepackage{amssymb}
\usepackage{amsthm}
\usepackage{xspace}

\usepackage{euscript}
\input xy
\xyoption{all} \CompileMatrices

\begin{document}

\newtheorem{theorem}{Theorem}[section]
\newtheorem{proposition}[theorem]{Proposition}
\newtheorem{corollary}[theorem]{Corollary}
\newtheorem{remark}[theorem]{Remark}
\newtheorem{lemma}[theorem]{Lemma}
\newtheorem{example}[theorem]{Example}
\newtheorem{conjecture}[theorem]{Conjecture}
\newtheorem{unnumblemma}{Lemma}[section]
\newtheorem{definition}{Definition}[section]


\def\XX{{\bf{X}}}
\def\YY{{\bf{Y}}}
\def\EE{{\bf{E}}}
\def\Int{{\operatorname{Int}}}
\def\rad{{\operatorname{rad}}}
\def\Spec{\operatorname{Spec}}
\def\Sym{\operatorname{Sym}}
\def\Hom{\operatorname{Hom}}
\def\TM{t\textup{-Max}}
\def\TS{t\textup{-Spec}}
\def\WBF{\operatorname{WBF}}
\def\sKr{\operatorname{sKr}}
\def\Ass{\operatorname{Ass}}
\def\wAss{\operatorname{wAss}}
\def\lwAss{\operatorname{LwAss}}
\def\ann{\operatorname{ann}}
\def\Cl{\operatorname{Cl}}
\def\ZZ{{\mathbb Z}}
\def\CC{{\mathbb C}}
\def\NN{{\mathbb N}}
\def\RR{{\mathbb R}}
\def\QQ{{\mathbb Q}}
\def\FF{{\mathbb F}}
\def\OO{{\mathcal O}}
\def\mm{{\mathfrak m}}
\def\nn{{\mathfrak n}}
\def\aaa{{\mathfrak a}}
\def\bbb{{\mathfrak b}}
\def\ppp{{\mathfrak p}}
\def\qqq{{\mathfrak q}}
\def\rrr{{\mathfrak r}}
\def\PPP{{\mathfrak P}}
\def\MM{{\mathfrak M}}
\def\qq{{\mathfrak Q}}
\def\rr{{\mathfrak R}}
\def\DD{{\mathfrak D}}
\def\cc{{\mathfrak S}}
\def\TT{{\mathcal{T}}}
\def\SS{{\mathcal S}}
\def\UU{{\mathcal U}}
\def\aa{{\bf{a}}}
\def\bb{{\bf{b}}}
\def\rad{\operatorname{rad}}
\def\End{\operatorname{End}}
\def\id{\operatorname{id}}
\def\mod{\operatorname{mod}}
\def\im{\operatorname{im}}
\def\ker{\operatorname{ker}}
\def\coker{\operatorname{coker}}
\def\ord{\operatorname{ord}}
\def\Bin{\operatorname{Bin}}
\def\eval{\operatorname{eval}}

\title[Integer-valued polynomial rings]{Integer-valued polynomial rings, $t$-closure, and associated primes}
\date{\today} \author{Jesse Elliott} \address{Department of Mathematics\\ California
State University, Channel Islands\\ Camarillo, California 93012}
\email{jesse.elliott@csuci.edu}

\begin{abstract}
Given an integral domain $D$ with quotient field $K$, the ring of {\it integer-valued polynomials on $D$} is the subring $\{f(X) \in K[X]: f(D) \subset D\}$ of the polynomial ring $K[X]$.  Using the tools of $t$-closure and associated primes, we generalize some known results on integer-valued polynomial rings over Krull domains, PVMD's, and Mori domains.  
\end{abstract}





\maketitle

\section{Introduction}

\subsection{Summary}

For any integral domain $D$ with quotient field $K$, any set $\XX$, and any subset $\EE$ of $K^\XX$, the ring of {\it integer-valued polynomials on $\EE$} is the subring
$$\Int(\EE,D) = \{f(\XX) \in K[\XX]: f(\EE) \subset D\}$$
of the polynomial ring $K[\XX]$.   One writes $\Int(D^\XX) = \Int(D^\XX,D)$, one writes $\Int(D^n) = \Int(D^{\{X_1,X_2, \ldots, X_n\}})$ for any positive integer $n$, and one writes $\Int(D) = \Int(D^1) = \Int(D,D)$.  Integer-valued polynomial rings possess a rich algebraic theory with strong connections to algebraic number theory, and at the same time they provide an ample source of examples and counterexamples in the theory of non-Noetherian commutative rings.  For example, the ring $\Int(\ZZ)$ is a simple and natural example of a non-Noetherian Pr\"ufer domain of Krull dimension two contained in $\QQ[X]$.   Moreover, the prime ideals of $\Int(\ZZ)$ containing a given prime number $p$ are in bijective correspondence with the $p$-adic integers, where $\alpha \in \ZZ_p$ corresponds to the prime ideal $\mm_{p, \alpha} = \{f(X) \in \Int(\ZZ): f(\alpha) \in p\ZZ_p\}$, and the prime ideal $\mm_{p,\alpha}$ has height 2 or 1 according as $\alpha \in \ZZ_p$ is algebraic or transcendental over $\QQ$ \cite[Proposition V.2.7]{cah}.

The study of integer-valued polynomial rings began around 1919 when P\'olya and Ostrowski characterized those number rings $\OO$ for which $\Int(\OO)$ has a free $\OO$-module basis consisting of exactly one polynomial of each degree.  (This holds, for example, if $\OO$ is a PID.)  As with ordinary polynomial rings, their study involves a wide range of techniques of commutative algebra. However, unlike polynomial rings, integer-valued polynomial rings tend to elude many of the standard techniques.  For example, since $\Int(D)_\ppp$ is not necessarily equal to $\Int(D_\ppp)$ for a prime ideal $\ppp$ of an arbitrary domain $D$, even the technique of localization has its limitations here.  Consequently, there are numerous open questions concerning $\Int(D)$, regarding, for example, its $D$-module structure, prime spectrum, and Picard group.  It is unknown, for instance, if there is a domain $D$ for which $\Int(D)$ is not free (or flat) as a $D$-module.  Also, it is unknown whether or not $\Int(D)$ has finite Krull dimension whenever $D$ has finite Krull dimension.  


Although many results about ordinary polynomial rings do not generalize to integer-valued polynomial rings, some tools from the theory of non-Noetherian commutative rings have found application here.  One such tool is that of a star operation, and in particular the $t$-closure star operation.  A {\it star operation} on an integral domain $D$ is a closure operator $I \longmapsto I^*$ on the partially ordered set of nonzero fractional ideals of $D$ that respects principal ideals in the sense that $I^* = I$ and $(IJ)^* = I J^*$ for any nonzero fractional ideals $I$ and $J$ such that $I$ is principal. (The condition on principal ideals is required to allow one to define the {\it star class group} $\Cl^*(D)$ of ``$*$-invertible'' $*$-closed ideals modulo principal ideals, which contains and generalizes the Picard group of $D$.)  The {\it $t$-closure} operation is one of the most useful examples of a star operation.  By definition it acts by $I \longmapsto I_t$, where $I_t = \bigcup (J^{-1})^{-1}$, where the union ranges over the set of all finitely generated ideals $J$ contained in $I$, and where $J^{-1}$ denotes the fractional ideal $(D:_K J)$, where $K$ is the quotient field of $D$.  One of its many uses is that an integral domain $D$ is a UFD if and only if the $t$-closure of any nonzero fractional ideal of $D$ is principal.

The articles \cite{clt,ell4,par,tar} have hinted at the applicability of the $t$-closure star operation to the study of integer-valued polynomial rings.  This paper further advances that theme.  In particular, we use the related tools of $t$-closure and associated primes to generalize some known results on integer-valued polynomial rings over Krull domains, PVMD's, and Mori domains, including \cite[Proposition 2.1]{cgh}, \cite[Theorem 1.3 and Propositions 6.8 and 6.10]{ell}, and \cite[Theorem 3.8 and Propositions 3.5 and 4.3]{ell4}.

A {\it $t$-ideal} of a domain $D$ is a nonzero fractional ideal $I$ such that $I = I_t$, and a {\it $t$-maximal} ideal is an ideal that is maximal among the $t$-ideals properly contained in $D$.  Every $t$-maximal ideal is prime.  Any invertible fractional ideal, for example, is a $t$-ideal, and therefore any nonzero fractional ideal of a Dedekind domain is a $t$-ideal; and the $t$-maximal ideals of a Krull domain are precisely the prime ideals of height one.  In Sections 1.2--1.4 we collect some definitions and facts about $t$-ideals and associated primes, TV domains \cite{hou}, H domains \cite{gla}, PVMD's, and $t$-linked extensions.  Few of the results in those sections are new but are included for the uninitiated reader.


In Section 2, we give some results that highlight the importance of the prime $t$-ideals and the weak Bourbaki associated primes with regard to integer-valued polynomial rings.  
For example, 
we prove in Section 2 that $\Int(S^{-1}D) = S^{-1}\Int(D)$ for every multiplicative subset $S$ of $D$ if $D$ is a domain {\it of finite $t$-character}, that is, if every nonzero element of $D$ is contained in only finitely many $t$-maximal ideals of $D$.  Note that Noetherian domains, Krull domains, Mori domains, TV domains, and domains of Krull type, for example, are all of finite $t$-character.  In Section 2 we also prove the following more general result.

\begin{theorem}\label{locallyfinite1}
Let $D$ be a domain that is equal to $\bigcap_{\ppp \in \SS} D_\ppp$ for some subset $\SS$ of $\Spec(D)$ such that every nonzero element of $D$ lies in only finitely many $\ppp$ in $\SS$.  Then $\Int(S^{-1}D) = S^{-1}\Int(D)$ for every multiplicative subset $S$ of $D$.
\end{theorem}

We also give some evidence in Section 2 for the claim that $\Int(D)$ is not flat over $D$ if $D = \FF_2[[T^2,T^3]]$ or if $D = \FF_2+T\FF_4[[T]]$.

Section 3 contains the main results of this paper.  To state our main theorem we need to give a little more background on integer-valued polynomial rings.

Unlike the situation with ordinary polynomial rings, if $D$ and $A$ are integral domains such that $D \subset A$, then it does not follow that $\Int(D) \subset \Int(A)$.  This leads one to study various properties of extensions of domains with regard to integer-valued polynomial rings.  One of the most important of these properties is the following.  As in \cite{ell}, we say that extension $A \supset D$ of domains is {\it polynomially regular} if $\Int(D,A)$ is generated by $\Int(D)$ as an $A$-module.  
For example, it is well-known that every extension of a Dedekind domain is polynomially regular.  More generally, by \cite[Corollary 3.14]{ell}, every flat extension of a Krull domain is polynomially regular, whereas the extension $\ZZ[T/2]$ of the UFD $\ZZ[T]$ is not, by \cite[Example 7.3]{ell}.  Several important and well-known conditions regarding integer-valued polynomial rings can be subsumed under the polynomial regularity condition.  For example, one has $\Int(S^{-1}D) = S^{-1}\Int(D)$ for a given multiplicative subset $S$ of $D$ if and only if $S^{-1}D$ is a polynomially regular extension of $D$.

A subset $E$ of $D$ is said to be a {\it polynomially dense subset} of $D$ if $\Int(E,D) = \Int(D)$.  Various authors have sought to characterize the polynomially dense subsets of a given domain (or class of domains).  Alternatively, authors have sought to characterize those domains containing a given domain as a polynomially dense subset.  In \cite{ell}, we called such extensions {\it polynomially complete} (for lack of a better term).  Thus, an extension $A$ of a domain $D$ is polynomially complete if and only if $\Int(D,A) = \Int(A)$.  This condition implies but is not equivalent to $\Int(D) \subset \Int(A)$.  As in \cite{ell}, if $\Int(D) \subset \Int(A)$, then we say that the extension $A$ of $D$ is {\it weakly polynomially complete}.  It is easy to verify that a polynomially regular extension is polynomially complete if and only if it is weakly polynomially complete.  Moreover, by \cite[Proposition 2.4]{ell}, the extension $\Int(D^\XX)$ of $D$ is the free polynomially complete extension of $D$ generated by $\XX$ for any infinite integral domain $D$.

Generally, an extension $A \supset D$ of domains is said to be {\it $t$-linked} if $I_t = D$ implies $(IA)_t = A$ for any nonzero ideal $I$ of $D$, or equivalently if $A = \bigcap_{\ppp \in \TM(D)} A_\ppp$, where $\TM(D)$ denotes the set of $t$-maximal ideals of $D$ \cite{dob}.  For example, any flat extension of a domain $D$ is $t$-linked, and the extension $\Int(D^\XX)$ of $D$ is $t$-linked for any set $\XX$.  (See Proposition \ref{tlinkedlemma} and Lemma \ref{inttlinked}.)  An extension $A \supset D$ of domains is said to be {\it unramified at $\ppp$}, where $\ppp$ is a prime ideal of $D$, if $\ppp A_\qqq = \qqq A_\qqq$ and $\kappa(\qqq) = A_\qqq/\qqq A_\qqq$ is a finite separable field extension of $\kappa(\ppp) = D_\ppp/\ppp D_\ppp$ for every prime ideal $\qqq$ of $A$ lying over $\ppp$.  Also, we say that $A$ has {\it trivial residue field extensions at $\ppp$} if $\kappa(\qqq) = \kappa(\ppp)$ for every prime ideal $\qqq$ of $A$ lying over $\ppp$.  The following result is our main theorem, proved in Section 3.

\begin{theorem}\label{introthm}
Let $D$ be an infinite domain such that $\ppp D_\ppp$ is principal and $\Int(D_\ppp) = \Int(D)_\ppp$ for every $t$-maximal ideal $\ppp$ of $D$ with finite residue field.  This holds, for example, if $D$ is a TV PVMD, or more generally if $D$ is an H PVMD such that $\Int(D_\ppp) = \Int(D)_\ppp$ for all $\ppp$.  Then we have the following.
\begin{enumerate}
\item[(a)] $\Int(D)$ is locally free as a $D$-module.
\item[(b)] $\Int(S^{-1}D) = S^{-1}\Int(D)$ for every multiplicative subset $S$ of $D$.
\item[(c)] Every $t$-linked extension of $D$ is polynomially regular.
\item[(d)] A $t$-linked extension of $D$ is polynomially complete if and only if it is unramified and has trivial residue field extensions at every $t$-maximal ideal of $D$ with finite residue field.  
\end{enumerate}
\end{theorem}

If $D$ is a Krull domain, then $\TM(D)$ is equal to the set $X^1(D)$ of prime ideals of $D$ of height one, and therefore an extension $A$ of $D$ is $t$-linked if and only if $A = \bigcap_{\ppp \in X^1(D)} A_\ppp$.  Thus Theorem \ref{introthm} generalizes \cite[Theorem 1.3]{ell} and \cite[Theorem 3.8]{ell4}.

Note that the integral domain $D = k[[T^2,T^3]]$, where $k$ is any finite field, does not satisfy the hypotheses of Theorem \ref{introthm}, since the unique maximal ideal $(T^2,T^3) = (T^2:_D T^3)$ of $D$ has residue field $k$ and is $t$-maximal but not principal.  We conjecture that this domain does not satisfy statements (a), (c), or (d) of Theorem \ref{introthm}.

By \cite[Proposition XI.1.1]{cah}, for any domain $D$ other than a finite field, one has $\Int(\Int(D^\XX)^\YY) = \Int(D^{\XX \amalg \YY})$ for any sets $\XX$ and $\YY$, in analogy with ordinary polynomial rings.  Likewise, for any domain $D$ and any set $\XX$ there exists a canonical $D$-algebra homomorphism $\theta_\XX: \bigotimes_{X \in \XX} \Int(D) \longrightarrow \Int(D^\XX)$, where the (possibly infinite) tensor product is over $D$.  In analogy with ordinary polynomial rings, one might hope, if not expect, that this homomorphism be an isomorphism.  However, no proof is known that $\theta_\XX$ is always an isomorphism, nor is there a known counterexample.  If $D$ is a domain such that the homomorphism $\theta_\XX$ is an isomorphism for any set $\XX$, then we say that $D$ is {\it polynomially composite}.  For example, if statements (a) and (b) of Theorem \ref{introthm} hold, then $D$ is polynomially composite by \cite[Proposition 6.10]{ell}.  In Section 3 we also give an application of our results to polynomial compositeness.

\subsection{$t$-ideals and associated primes}

For the remainder of Section 1 we review some definitions and facts about $t$-ideals and associated primes, TV domains, H domains, PVMD's, and $t$-linked extensions.  Throughout $D$ is an integral domain with quotient field $K \neq D$.

A $D$-submodule $I$ of $K$ is a {\it fractional ideal} of $D$ if $dI \subset D$ for some nonzero $d \in D$.  For any fractional ideal $I$ of $D$ let $I^{-1} = (D :_K I)$, let $I_v = (I^{-1})^{-1}$, and let $I_t = \bigcup J_v$, where $J$ ranges over the set of nonzero finitely generated fractional ideals contained in $I$.  The operations $I \longmapsto I_v$ and $I \longmapsto I_t$ are examples of {\it closure operators} $*: I \longmapsto I^*$ on the partially ordered set of nonzero fractional ideals of $D$, in that $I \subset I^*$ and $(I^*)^* = I^*$, and $I \subset J$ implies $I^* \subset J^*$, for all $I$ and $J$.  Moreover, these two closure operators {\it respect principal fractional ideals} in that $I^* = I$ and $(IJ)^* = I J^*$ if $I$ is principal.  A closure operator on the partially ordered set of nonzero fractional ideals of $D$ that respects principal fractional ideals is called a {\it star operation on $D$}.
For any star operation $*$ on $D$, one has $I^* \subset I_v$ for any nonzero fractional ideal $I$ of $D$.  Thus the $v$-closure operation is the ``coarsest'' of all star operations.

A fractional ideal $I$ of $D$ is said to be a {\it $t$-ideal} if $I = I_t$ and a {\it $v$-ideal} if $I = I_v$.  
Every invertible fractional ideal, for example, is a $v$-ideal.  Also, every $v$-ideal is a $t$-ideal, and every finitely generated $t$-ideal is a $v$-ideal.  If every $t$-ideal of $D$ is a $v$-ideal, then $D$ is said to be a {\it TV domain}.

An ideal $I$ of $D$ is {\it $t$-prime} if $I$ is a prime $t$-ideal, and $I$ is {\it $t$-maximal} if it is maximal among the $t$-ideals properly contained in $D$.  The {\it $v$-prime} and {\it $v$-maximal} ideals are defined similarly.  The $t$-maximal and $v$-maximal ideals of a domain are all prime.  By an application of Zorn's lemma, every nonunit, and in fact every proper $t$-ideal, of $D$ is contained in some $t$-maximal ideal of $D$.  By contrast, a domain may have no $v$-maximal ideals.  The set of $t$-prime and $t$-maximal ideals of $D$ will be denoted $\TS(D)$ and $\TM(D)$, respectively.

Unfortunately, $t$-closure does not in general commute with localization.  However, if $I$ is a $t$-ideal of $D_\ppp$ for some prime ideal $\ppp$ of $D$, then $I \cap D$ is a $t$-ideal of $D$ \cite{kang}.  Thus, if a prime ideal $\ppp$ of $D$ is {\it $t$-localizing}, that is, if $\ppp D_\ppp$ is a $t$-ideal of $D_\ppp$, then $\ppp$ is a $t$-ideal of $D$.  Thus, for example, any nonzero prime $\ppp$ such that $\ppp D_\ppp$ is principal is $t$-localizing.

A prime ideal $\ppp$ of $D$ is said to be an {\it associated prime} (in the Bourbaki sense) of a $D$-module $M$ if $\ppp$ equals the annihilator of some element of $M$.
A prime ideal $\ppp$ is said to be a {\it weakly associated prime} of $M$ if $\ppp$ is minimal over the annihilator of some element of $M$.  The sets of all such primes $\ppp$ of $D$ are denoted $\Ass(M)$ and $\wAss(M)$, respectively.  If $M \neq 0$, then $\wAss(M) \neq \emptyset$, but $\Ass(M)$ may be empty.  A {\it conductor ideal} of $D$ is an ideal of the form $(aD:_D bD)$ with $a, b \in D$ and $a \neq 0$.  A prime ideal $\ppp$ of $D$ lies in $\wAss(K/D)$ (resp., $\Ass(K/D)$) if and only if $\ppp$ is minimal over (resp., equal to) some conductor ideal of $D$.  Note that $D = \bigcap_{\ppp \in \wAss(K/D)} D_\ppp$ for any domain $D$.

A prime ideal $\ppp$ of $D$ is said to be a {\it strong Krull prime}, or {\it Northcott attached prime}, of a $D$-module $M$ if for every finitely generated ideal
$I \subset \ppp$ there exists an $m \in M$ such that $I \subset \ann(m) \subset \ppp$ \cite{dut}.  We let $\sKr(M)$ denote the set of all strong Krull primes of $D$.  One has $\wAss(M) \subset \sKr(M)$ for any $D$-module $M$ by \cite[Proposition 2]{dut}, and the finitely generated ideals of $\Ass(M)$ and $\sKr(M)$ coincide.

By \cite[Proposition 1.1]{zaf1}, a prime ideal $\ppp$ of $D$ is $t$-localizing if and only if $\ppp \in \sKr(K/D)$, that is, if and only if for every finitely generated ideal $I \subset \ppp$ there exists a conductor ideal $J$ of $D$ such that $I \subset J \subset \ppp$.  In particular, one has $$\Ass(K/D) \subset \wAss(K/D) \subset \sKr(K/D) \subset \TS(D).$$

\subsection{TV domains, H domains, and PVMD's}

An ideal $I$ of $D$ is said to be {\it $t$-invertible} if $(I^{-1}I)_t = D$ and {\it $v$-invertible} if $(I^{-1}I)_v = D$.  The following result follows from \cite[Theorem 2.2 and Proposition 4.4]{zaf}.

\begin{proposition}\label{tinv}
Let $D$ be an integral domain with quotient field $K$.  Every $t$-invertible $t$-prime of $D$ is $t$-maximal and lies in $\Ass(K/D)$.  Moreover, a $t$-maximal ideal $\ppp$ of $D$ is $t$-invertible
if and only if $\ppp D_\ppp$ is principal and $\ppp = I_v$ for some finitely generated ideal $I$ of $D$.
\end{proposition}

Every $t$-invertible ideal of $D$ is $v$-invertible.  If the converse holds, then $D$ is said to be an {\it H domain}.  We have the following implications: Noetherian $\Rightarrow$ Mori $\Rightarrow$ TV $\Rightarrow$ H.  By \cite[Proposition 4.2]{zaf}, we have the following.

\begin{proposition}\label{Hdomains}
For any domain $D$, the following conditions are equivalent.
\begin{enumerate}
\item $D$ is an H domain, that is, every $v$-invertible ideal of $D$ is $t$-invertible.
\item Every $t$-maximal ideal of $D$ is $v$-maximal.
\item Every $t$-maximal ideal of $D$ is a conductor ideal of $D$.
\item For any nonzero ideal $I$ of $D$, one has $I_t = D$ if and only if $I_v = D$.
\end{enumerate}
\end{proposition} 

A domain $D$ is said to be a {\it Pr\"ufer $v$-multiplication domain}, or {\it PVMD}, if every nonzero finitely generated ideal of $D$ is $t$-invertible, or equivalently (as is well-known) if $D_\ppp$ is a valuation domain for every $t$-maximal ideal $\ppp$ of $D$.  For example, any Krull domain or Pr\"ufer domain is a PVMD.  A domain $D$ is said to be {\it of finite $t$-character} if every nonzero element of $D$ is contained in only finitely many $t$-maximal ideals of $D$.  By \cite[Theorem 1.3 and Proposition 2.4]{hou}, every TV domain is an H domain of finite $t$-character.  (We do not know if the converse is true.)  

\begin{proposition}\label{HPVMD1} A PVMD $D$ is an H domain if and only if every $t$-maximal ideal of $D$ is $t$-invertible.
\end{proposition}

\begin{proof}
Suppose that $D$ is an H domain, and let $\ppp$ be a $t$-maximal ideal of $D$.  Then $\ppp = (aD:_D bD) = (D+\frac{b}{a}D)^{-1}$ for some nonzero elements $a$ and $b$ of $D$.  Since the fractional ideal $D+\frac{b}{a}D$ of the PVMD is finitely generated, it is $t$-invertible.  It follows that $\ppp = (D+\frac{b}{a}D)^{-1}$ is $t$-invertible as well.  The converse follows from the fact that every $t$-invertible $t$-prime ideal is a conductor ideal, which follows from Proposition \ref{tinv}.
\end{proof}

\subsection{$t$-linked extensions}

As noted in the introduction, an extension $A$ (not necessarily an overring) of a domain $D$ is said to be {\it $t$-linked} if $I_t = D$ implies $(IA)_t = A$ for any nonzero ideal $I$ of $D$.   The proofs of \cite[Propositions 2.1 and 2.13]{dob} yield the following equivalent characterizations of the $t$-linked extensions of a domain $D$.  (Alternatively, see \cite[Propositions 3.4, 3.5, and 4.6]{ell2}.)

\begin{proposition}\label{tlinkedextensions}
For any extension $A \supset D$ of domains, the following conditions are equivalent.
\begin{enumerate}
\item $A$ is a $t$-linked extension of $D$.
\item For every $\qqq \in \TS(A)$ one has $\qqq \cap D = 0$ or $(\qqq \cap D)_t \neq D$.
\item For every $\qqq \in \TM(A)$ one has $\qqq \cap D \subset \ppp$ for some $\ppp \in \TM(D)$.
\item $A = \bigcap_{\ppp \in \TM(D)} A_\ppp$.
\end{enumerate}
\end{proposition}

Next, we note that flat extensions are $t$-linked.  In fact, we have the following.

\begin{proposition}\label{tlinkedlemma}
Let $A$ be a $t$-linked extension of an integral domain $D$.  Then we have the following.
\begin{enumerate}
\item[(a)] If $B$ is a $t$-linked extension of $A$, then $B$ is a $t$-linked extension of $D$.
\item[(b)] If $B$ is a flat extension of $D$, then both $B$ and $A \otimes_D B$ are $t$-linked extensions of $D$.
\end{enumerate}
\end{proposition}

\begin{proof}
Statement (a) is clear from Proposition \ref{tlinkedextensions}.  If $B$ is a flat extension of $D$, it follows from Remark (a) following \cite[Proposition 2.6]{zaf} that $B$ is a $t$-linked extension of $D$.  Moreover, $A \otimes_D B$ is a flat and therefore $t$-linked extension of $A$.  Thus, since $A$ is a $t$-linked extension of $D$, it follows from (a) that $A \otimes_D B$ is a $t$-linked extension of $D$.
\end{proof}

An extension $A$ of a domain $D$ is {\it locally $t$-linked} if $A_\ppp$ is a $t$-linked extension of $D_\ppp$ for every prime ideal $\ppp$ of $D$.  For example, any flat extension of domains is locally $t$-linked.  By \cite[Corollary 4.11]{ell2}, any locally $t$-linked extension is $t$-linked.  However, the converse is not true.  For example, by \cite[Example 4.12]{ell2}, any valuation overring of the integral domain $A = \QQ+(X,Y,Z)\QQ(\sqrt{2})[[X,Y,Z]]$ centered on the prime ideal $(X,Y)\QQ(\sqrt{2})[[X,Y,Z]]$  is a $t$-linked but not a locally $t$-linked extension of $A$. 

By \cite[Corollary 4.11]{ell2}, we have the following.

\begin{proposition}\label{locallytlinked}
Let $D$ be an integral domain with quotient field $K$.  For any extension $A \supset D$ of domains, the following conditions are equivalent.
\begin{enumerate}
\item $A$ is a locally $t$-linked extension of $D$.
\item For every $t$-localizing prime $\qqq$ of $A$, either $\qqq \cap D = 0$ or $\qqq \cap D$ is a $t$-localizing prime of $D$.
\item $A_\ppp = \bigcap_{\ppp \supset \qqq \in \sKr(K/D)} A_\qqq$ for every prime ideal $\ppp$ of $D$.
\end{enumerate}
Moreover, these conditions imply that $A = \bigcap_{\ppp \in \sKr(K/D)} A_\ppp$ and therefore that $A$ is a $t$-linked extension of $D$.
\end{proposition}

\section{Local properties of $\Int(D)$}

\subsection{Int primes}

A prime ideal $\ppp$ of a domain $D$ is said to be a {\it polynomial prime} of $D$ if $\Int(D) \subset D_\ppp[X]$, or equivalently if $\Int(D)_\ppp = D_\ppp[X]$; otherwise $\ppp$ is said to be an {\it int prime} of $D$ \cite{clt}.  Let us say that a prime ideal $\ppp$ of $D$ is a {\it strong polynomial prime} of $D$ if $\Int(D_\ppp) = D_\ppp[X]$; otherwise we will say that $\ppp$ is a {\it weak int prime}.  For example, every prime ideal with infinite residue field is a strong polynomial prime, by \cite[Corollary I.3.7]{cah}.  Note that $\ppp$ is a strong polynomial prime of $D$ if and only if $\ppp D_\ppp$ is a polynomial prime of $D_\ppp$.  By \cite[Proposition I.2.2]{cah} one has $D_\ppp[X] \subset \Int(D)_\ppp \subset \Int(D_\ppp)$ for any prime $\ppp$.  Thus every strong polynomial prime is a polynomial prime, but the converse is not true, as shown by \cite[Example 5.3]{clt}.  Moreover, it follows that $\ppp$ is a strong polynomial prime of $D$ if and only if $\ppp$ is a polynomial prime of $D$ and $\Int(D_\ppp) = \Int(D)_\ppp$.  We therefore have the following.

\begin{lemma}\label{strongpolynomialprime}
Let $D$ be an integral domain, and let $\ppp$ be a prime ideal of $D$.  The following are equivalent.
\begin{enumerate}
\item $\ppp$ is a strong polynomial prime of $D$.
\item $\ppp D_\ppp$ is a polynomial prime of $D_\ppp$.
\item $\ppp$ is a polynomial prime of $D$ and $\Int(D_\ppp) = \Int(D)_\ppp$.
\end{enumerate}
\end{lemma}

The following result follows from \cite[Proposition 3.3]{ell4} (the proof of which was based on \cite[Theorem 1.5]{rus} and \cite[Proposition 1.2]{clt}).

\begin{lemma}\label{implications}
Let $D$ be an integral domain with quotient field $K$.  Then, for any prime $\ppp$ of $D$ with finite residue field, we have the following implications.
\begin{eqnarray*}
\SelectTips{cm}{11}\xymatrix{
{\ppp \mbox{ principal}} \ar@{=>}[d] \ar@{=>}[r] & {\ppp D_\ppp \mbox{ principal}} \ar@{=>}[d]   \\
{\ppp \in \Ass(K/D)} \ar@{=>}[r] \ar@{=>}[d]  & {\ppp D_\ppp \in \Ass(K/D_\ppp) \ar@{=>}[d] } \\
{\ppp \mbox{ int prime of }D} \ar@{=>}[d] \ar@{=>}[r] & {\ppp \mbox{ weak int prime of }D} \ar@{=>}[d]   \\
{\ppp \in \wAss(K/D)} \ar@{=>}[r] \ar@{=>}[d]  & {\ppp D_\ppp \in \wAss(K/D_\ppp) \ar@{=>}[d] }  \\
{\ppp \in \sKr(K/D)} \ar@{<=>}[r] \ar@{=>}[d] & {\ppp D_\ppp \in \sKr(K/D_\ppp) \ar@{<=>}[d] }  \\
{\ppp \in \TS(D)}  & {\ppp D_\ppp \in \TS(D_\ppp) \ar@{=>}[l]} 
}
\end{eqnarray*}
\end{lemma}




\subsection{The condition $\Int(S^{-1}D) = S^{-1}\Int(D)$}

For some classes of domains $D$, such as the Noetherian domains, and more generally the Mori domains, by \cite[Proposition 2.1]{cgh}, the equality $\Int(S^{-1}D) = S^{-1}\Int(D)$ is known to hold for all multiplicative subsets $S$ of $D$.  The following result gives some characterizations of those domains $D$ for which the equality always holds.  First, recall that an extension $A$ of a domain $D$ is said to be {\it polynomially regular} if $\Int(D,A)$ is generated by $\Int(D)$ as an $A$-module.  

\begin{proposition}\label{localization}
For any integral domain $D$ with quotient field $K$, the following conditions are equivalent.
\begin{enumerate}
\item Every flat overring of $D$ is polynomially regular.
\item $\Int(S^{-1}D) = S^{-1}\Int(D)$ for every multiplicative subset $S$ of $D$.
\item $\Int(D_\ppp) = \Int(D)_\ppp$ for every prime ideal $\ppp$ of $D$.
\item $\Int(D_\ppp) = \Int(D)_\ppp$ for every $t$-localizing $t$-maximal ideal $\ppp$ of $D$ with finite residue field.
\item $\Int(D_\ppp) = \Int(D)_\ppp$ for every prime ideal $\ppp$ of $D$ such that $\ppp D_\ppp$ is a weakly associated prime of $K/D_\ppp$ with finite residue field.
\item $\Int(D_\ppp) = \Int(D)_\ppp$ for every weak int prime $\ppp$ of $D$.
\end{enumerate}
\end{proposition}

\begin{proof}
By \cite[Corollary I.2.6]{cah} one has $\Int(D, S^{-1}D) = \Int(S^{-1}D)$, from which it follows that (1) implies (2).  The implications $(2) \Rightarrow (3) \Rightarrow (4)$ are trivial.  By Lemma \ref{implications} one has $(4) \Rightarrow (5) \Rightarrow (6)$.  Moreover, (6) implies (3) by Lemma \ref{strongpolynomialprime}.  Suppose that (3) holds.  Let $D'$ be a flat overring of $D$, and let $\ppp'$ be a maximal ideal of $D'$.  Then $D'_{\ppp'} = D_\ppp$, where $\ppp = \ppp' \cap D$.  Thus we have $\Int(D,D')_{\ppp'} \subset \Int(D,D'_{\ppp'}) = \Int(D,D_\ppp) = \Int(D_\ppp) = \Int(D)_\ppp = (D'\Int(D))_{\ppp'}$, where $D'\Int(D)$ denotes the $D'$-module generated by $\Int(D)$.  Therefore we have $\Int(D,D')_{\ppp'} = (D'\Int(D))_{\ppp'}$ for every maximal ideal $\ppp'$ of $D'$.  It follows that $\Int(D,D') = D' \Int(D)$ and thus $D'$ is a polynomially regular
extension of $D$.  Thus (3) implies (1), and the six conditions are equivalent.
\end{proof}

If the equivalent conditions of Proposition \ref{localization} hold, then we will say that $D$ is {\it polynomially L-regular}.  Thus, for example, any Mori domain is polynomially L-regular.  We will generalize this in Proposition \ref{tfinitereg} and Theorem \ref{locallyfinite} below.

A domain $D$ is said to be {\it of finite character} if every nonzero element of $D$ lies in only finitely many maximal ideals.  For example, any Dedekind domain or semilocal domain is of finite character.  A domain $D$ is said to be {\it of finite $t$-character} if every nonzero element of $D$ lies in only finitely many $t$-maximal ideals.  For example, every TV domain is of finite $t$-character.

\begin{proposition}\label{tfinitereg}
Every domain of finite $t$-character, and in particular every TV domain, is polynomially L-regular.  Likewise, every domain of finite character is polynomially L-regular.
\end{proposition}

\begin{proof}
By \cite[Proposition 2.3]{tar}, if $D$ is of finite $t$-character, then condition (4) of Proposition \ref{localization} holds, and if $D$ is of finite character, then condition (3) of Proposition \ref{localization} holds.
\end{proof}

We may generalize both Proposition \ref{tfinitereg} and \cite[Proposition 2.3]{tar}.  First, we prove the following somewhat surprising lemma, which holds even for domains that are not polynomial L-regular.

\begin{lemma}\label{interchange}
One has $(\Int(D_\ppp))_\qqq = (\Int(D_\qqq))_\ppp$ for any prime ideals $\ppp$ and $\qqq$ of an integral domain $D$, and in fact $(\Int(D_\ppp))_\qqq = (D_\ppp)_\qqq[X]$ if $\ppp \neq \qqq$.
\end{lemma}

\begin{proof}
The result is trivial if $\ppp = \qqq$.  Suppose $\ppp \neq \qqq$, and let
$\SS$ be the set of all primes of $D$ contained in $\ppp \cap \qqq$.  Since no prime of
$\SS$ is maximal one has $\Int(D_\qqq) \subset \Int(D_{\ppp'}) = D_{\ppp'}[X]$ for each prime $\ppp' \in \SS$.  Therefore $$\Int(D_\qqq) \subset \bigcap_{\ppp' \in \SS} D_{\ppp'}[X] = (D_\ppp)_\qqq[X] = (D_\ppp[X])_{\qqq} \subset (\Int(D_\ppp))_\qqq,$$ hence
$(\Int(D_\qqq))_\ppp \subset (D_\ppp)_\qqq[X] \subset (\Int(D_\ppp))_\qqq$ and the result follows by symmetry.
\end{proof}

An intersection $\bigcap_{i \in I} D_i$ of integral domains $D_i$ each contained in the same field is said to be {\it locally finite} if every nonzero element of the intersection is a unit in $D_i$ for all but finitely many $i \in I$.  For example, a domain $D$ is of finite $t$-character if and only if the intersection $D = \bigcap_{\ppp \in \TM(D)} D_\ppp$ is locally finite.

\begin{theorem}\label{locallyfinite}
Let $D$ be a domain that is equal to a locally finite intersection $\bigcap_{\ppp \in \SS} D_\ppp$ for some subset $\SS$ of $\Spec(D)$.  Then $\Int(S^{-1}D) = S^{-1}\Int(D)$ for every multiplicative subset $S$ of $D$, or, in other words, $D$ is polynomially L-regular.
\end{theorem}

\begin{proof}
Let $\qqq$ be any prime ideal of $D$.  Since $\Int(D)_\qqq \subset \Int(D, D_\qqq) = \Int(D_\qqq)$, we need only show $\Int(D_\qqq) \subset \Int(D)_\qqq$.  Let $f \in \Int(D_\qqq)$, and write $f = g/a$, where $g \in D[X]$ and $a \in D \backslash 0$.  Since the intersection $\bigcap_{\ppp \in \SS} D_\ppp$ is locally finite, one has $a \notin \ppp$, and therefore $f \in D_\ppp[X]$, for all but finitely many $\ppp \in \SS$, say, $\ppp_1, \ppp_2, \ldots, \ppp_n$.  By Lemma \ref{interchange}, for each $i$ one has $\Int(D_\qqq) \subset (\Int(D_{\ppp_i}))_\qqq$, so there exists $s_i \in D \backslash \qqq$ such that $s_i f \in \Int(D_{\ppp_i})$.  Let $s = s_1 s_2 \cdots s_n$.  Then
$s \in D \backslash \qqq$ and $s f \in \Int(D_{\ppp_i})$ for $i = 1, 2, \ldots, n$.
But $s f \in D_\ppp[X] \subset \Int(D_\ppp)$ for all $\ppp \in \SS \backslash \{\ppp_1, \ppp_2, \ldots, \ppp_n\}$.  Hence $sf \in \Int(D_\ppp)$ for all $\ppp \in \SS$.  It follows that $sf \in \Int(D)$, and thus $f \in \Int(D)_\qqq$.
This completes the proof. 
\end{proof}

Note that a domain $D$ is equal to a locally finite intersection $\bigcap_{\ppp \in \SS} D_\ppp$, where $\SS$ is a subset of $\Spec(D)$, if and only if every nonzero proper conductor ideal of $D$ is contained in a finite and nonzero number of prime ideals in $\SS$.

Regarding PVMD's, we record the following.

\begin{lemma}\label{TVPVMD}
Let $D$ be an integral domain.  Each of the following conditions implies the next.
\begin{enumerate}
\item[(a)] $D$ is a Krull domain.
\item[(b)] $D$ is a TV PVMD.
\item[(c)] $D$ is an H PVMD of finite $t$-character.
\item[(d)] $D$ is a polynomially L-regular H PVMD.
\item[(e)] $D$ is a polynomially L-regular PVMD such that $\ppp D_\ppp$ is principal for every $t$-maximal ideal $\ppp$ of $D$.
\end{enumerate}
\end{lemma}

\begin{proof}
Clearly (a) implies (b) and (b) implies (c), and (c) implies (d) by Proposition \ref{tfinitereg}.  Moreover, (d) implies (e) by Propositions \ref{HPVMD1} and \ref{tinv}.
\end{proof}

\subsection{$\Int(D)$ as an extension of $D$}

If $\ppp$ is a polynomial prime of $D$, then $\Int(D)_\ppp = D_\ppp[X]$ is free as a $D_\ppp$-module.  Thus by Lemma \ref{implications} we have the following.

\begin{lemma}\label{locallyfreeprop}
Let $D$ be an integral domain with quotient field $K$.  Then the following conditions are equivalent.
\begin{enumerate}
\item $\Int(D)$ is locally free (resp., flat) as a $D$-module.
\item $\Int(D)_\ppp$ is free (resp., flat) as a $D_\ppp$-module for every $t$-localizing prime $\ppp$ of $D$ with finite residue field.
\item $\Int(D)_\ppp$ is free (resp., flat) as a $D_\ppp$-module for every $\ppp \in \wAss(K/D)$ with finite residue field.
\item $\Int(D)_\ppp$ is free (resp., flat) as a $D_\ppp$-module for every int prime $\ppp$ of $D$.
\end{enumerate}
\end{lemma}

Note that, by \cite[Exercise II.16]{cah}, if $\ppp D_\ppp$ is principal, then $\Int(D_\ppp)$ is free as a $D_\ppp$-module.  Therefore Lemma \ref{locallyfreeprop} implies the following.

\begin{lemma}\label{locallyfreecor}
Let $D$ be an integral domain such that $\ppp D_\ppp$ is principal and $\Int(D_\ppp) = \Int(D)_\ppp$ for every $t$-localizing $t$-maximal ideal $\ppp$ of $D$ with finite residue field.  Then $\Int(D)$ is locally free as a $D$-module.
\end{lemma}

Remarkably, there are no known examples of domains $D$ such that $\Int(D)$ is not free as a $D$-module.  Below we make some progress on this problem by providing some evidence for the following conjecture.

\begin{conjecture}\label{conj}
There exists a local, Noetherian, one dimensional, analytically irreducible integral domain $D$ such that $\Int(D)$ is not flat as a $D$-module.
\end{conjecture}

\begin{lemma}\label{conjlem1}
Let $A \supset D$ be a flat extension of integral domains.  Then $(IA)^{-1} = I^{-1}A$ for every finitely generated fractional ideal $I$ of $D$.
\end{lemma}

\begin{proof}
This is well-known and follows readily from the fact that if $M$ is a flat module over a commutative ring $R$ then $IM \cap J M = (I \cap J) M$ for all ideals $I$ and $J$ of $R$. 
\end{proof}

For any fractional ideal $I$ of an integral domain $D$, we define $$\Int(D,I) = \{f(X) \in K[X]: f(D) \subset I\},$$ where $K$ is the quotient field of $D$.  This is a fractional ideal of $\Int(D)$.

\begin{lemma}\label{conjlem2}
Let $D$ be an integral domain.  Then $(I\Int(D))^{-1} = \Int(D,I^{-1}) = (\Int(D,I))^{-1}$ for any nonzero fractional ideal $I$ of $D$.
\end{lemma}

\begin{proof}
This is \cite[Lemma 4.1(a)]{cgh}.
\end{proof}

\begin{proposition}\label{mainprop}
Let $D$ be an integral domain for which there exists a finitely generated ideal $I$ of $D$ such that $D' = I^{-1}$ is an overring of $D$.  If $\Int(D)$ is flat over $D$, then $\Int(D,D') = D' \Int(D)$.
\end{proposition}

\begin{proof}
We have
\begin{eqnarray*}
\Int(D,D') & = & \Int(D,I^{-1}) \\
		& = & (I\Int(D))^{-1} \\
		& = & I^{-1}\Int(D) \\
		& = & D'\Int(D),
\end{eqnarray*}
where the second equality holds by Lemma \ref{conjlem2} and the third holds by Lemma \ref{conjlem1}.
\end{proof}

Note, for example, that if $D$ is a local domain with quotient field $K$ and maximal ideal $I = M$, then $M^{-1}$ is an overring of $D$ if and only if $M^{-1} = (M :_K M)$ if and only if $M$ is not principal.

\begin{proposition}\label{mainprop1}
Let $D' = \FF_4[[T]]$, let $M$ be the maximal ideal of $D'$, and let $D = \FF_2+M$.  Each of the following statements implies the next.
\begin{enumerate}
\item $(X^2+X)/T \notin D'\Int(D)$, or equivalently $X^2 + X \notin M\Int(D)$. 
\item $\Int(D,D') \neq D'\Int(D)$.
\item $\Int(D)$ is not flat as a $D$-module.
\end{enumerate}
\end{proposition}

\begin{proof}
Statement (1) implies statement (2) because $(X^2+X)/T \in \Int(D,D')$.
Next, $M$ is a finitely generated ideal of $D$ and one has $M^{-1} = \FF_4[[T]] = D'$, which is an overring of $D$.  By Proposition \ref{mainprop}, if $\Int(D)$ is flat over $D$, then $\Int(D,D') = D' \Int(D)$, so statement (2) implies statement (3).  
\end{proof}

\begin{conjecture}\label{conj1}
We conjecture that statement (2) (or statement (1)) of Proposition \ref{mainprop1} holds and therefore $\Int(D)$ is not flat as a $D$-module if $D = \FF_2 + T\FF_4[[T]]$.
\end{conjecture}


\begin{proposition}\label{mainprop2}
Let $D' = \FF_2[[T]]$, let $D = \FF_2[[T^2,T^3]]$, and let $M$ be the maximal ideal of $D$.  Each of the following statements implies the next.
\begin{enumerate}
\item $(X^2+X)/T^2 \notin D'\Int(D)$, or equivalently $X^2 + X \notin M \Int(D)$. 
\item $\Int(D,D') \neq D'\Int(D)$.
\item $\Int(D)$ is not flat as a $D$-module.
\end{enumerate}
\end{proposition}

\begin{proof}
The proof is similar to that of Proposition \ref{mainprop1}, since the maximal ideal $M = (T^2,T^3)$ of $D$ is finitely generated and satisfies $M^{-1} = D'$.
\end{proof}

\begin{conjecture}\label{conj2}
We conjecture that statement (2) (or statement (1)) of Proposition \ref{mainprop2} holds and therefore $\Int(D)$ is not flat as a $D$-module if $D = \FF_2[[T^2,T^3]]$.
\end{conjecture}

Since the domains $\FF_2 + T\FF_4[[T]]$ and $\FF_2[[T^2,T^3]]$ are local, Noetherian, one dimensional, and analytically irreducible, we obtain the following.

\begin{proposition}
Conjecture \ref{conj1} and Conjecture \ref{conj2} each imply Conjecture \ref{conj}.
\end{proposition}




An extension $A$ of a domain $D$ is said to be {\it $v$-linked} if $I_v = D$ implies $(IA)_v = A$ for any nonzero ideal $I$ of $D$.  Every $v$-linked extension of a domain $D$ is $t$-linked.  An extension $A$ of a domain $D$ is said to be {\it $t$-compatible} if $I_t \subset (IA)_t$ for every nonzero ideal $I$ of $D$ \cite[Section 2]{zaf}.  The {\it $v$-compatible} extensions are defined similarly.  By \cite[Proposition 2.6 Remark (a)]{zaf}, one has the implications
\begin{eqnarray*}
\SelectTips{cm}{11}\xymatrix{
 & {v\mbox{-compatible}} \ar@{=>}[d] \ar@{=>}[r] & {v\mbox{-linked}} \ar@{=>}[d] \\
 {\mbox{flat}} \ar@{=>}[r] & {t\mbox{-compatible}} \ar@{=>}[r] & {t\mbox{-linked}} }
\end{eqnarray*}
for any extension $A$ of a domain $D$.
Although we conjecture that $\Int(D)$ is not necessarily flat as a $D$-module, it is known that $\Int(D)$ is a $v$-compatible (hence $v$-linked, $t$-compatible, and $t$-linked) extension of $D$ \cite[Lemma 4.1(2)]{cgh}.  In fact, we have the following.

\begin{lemma}\label{inttlinked}
Let $D$ be an integral domain and let $\XX$ be a set.  Then we have the following.
\begin{enumerate}
\item[(a)] $\Int(D^\XX)$ is a $v$-compatible, hence $v$-linked, $t$-compatible, and $t$-linked, extension of $D$.
\item[(b)] If $D$ is polynomially L-regular, then $D'\Int(D)$ is a $v$-compatible extension of $D'$ for any flat overring $D'$ of $D$, and $\Int(D)$ is a locally $t$-linked extension of $D$.
\end{enumerate}
\end{lemma}

\begin{proof}
By \cite[Lemma 4.1(2)]{cgh}, $\Int(D)$ is a $v$-compatible extension of $D$.  The proof may be easily adapted to show that $\Int(D^\XX)$ is a $v$-compatible extension of $D$, and (a) follows.  If $D$ is polynomially L-regular and $D'$ is a flat overring of $D$, then $\Int(D') = \Int(D,D') = D'\Int(D)$ by \cite[Lemma 3.5]{ell}, and therefore $D'\Int(D) = \Int(D')$ is a $v$-compatible extension of $D'$ by (a).  It then follows that $\Int(D)$ is a locally $t$-linked extension of $D$.  Thus (b) also holds. 
\end{proof}

If $D$ is polynomially L-regular, then we do not know if $\Int(D^\XX)$ is necessarily locally $t$-linked over $D$ for any set $X$. However, this does hold if $\Int(D_\ppp^\XX) = \Int(D^\XX)_\ppp$ for all $X$.  By transfinite induction on $|X|$, the latter equality holds if $\Int(D)$ is polynomially L-regular whenever $D$ is.  Moreover, the equality holds if $D$ is {\it polynomially composite}, as defined at the end of Section 1.2, or if $D$ satisfies the hypotheses of Theorem \ref{locallyfinite}.

\section{Polynomial completeness and regularity conditions}

\subsection{Proof of main theorem}

In this section we prove Theorem \ref{introthm} of the introduction.

First, we note that the following theorem follows from \cite[Theorem 1.4 and Proposition 3.6]{ell}.

\begin{theorem}\label{mostgeneralthm}
Let $A$ be an extension of an infinite domain $D$ such that $A = \bigcap_{\ppp \in \SS} A_\ppp$, where $\SS$ is a set of prime ideals of $D$ such that $\ppp D_\ppp$ is principal and $\Int(D_\ppp) = \Int(D)_\ppp$ for every $\ppp \in \SS$ with finite residue field.  Then $\Int(D,A) = \bigcap_{\ppp \in \SS} (A\Int(D))_\ppp$, and the following conditions are equivalent.
\begin{enumerate}
\item $A$ is a polynomially complete extension of $D$.
\item $A_\ppp$ is a polynomially complete extension of $D_\ppp$ for every $\ppp \in \SS$ with finite residue field.
\item $A$ is unramified and has trivial residue field extensions at every $\ppp \in \SS$ with finite residue field.
\item $A$ is a weakly polynomially complete extension of $D$.
\item $A_\ppp$ is a weakly polynomially complete extension of $D_\ppp$ every $\ppp \in \SS$ with finite residue field.
\end{enumerate} 
\end{theorem}

Applying Theorem \ref{mostgeneralthm} to the set $\SS = \TM(D)$ of $t$-maximal ideals of $D$, we obtain the following.

\begin{corollary}\label{pvmdext}
Let $D$ be a domain such that every $t$-maximal ideal $\ppp$ of $D$ with finite residue field is locally principal and satisfies $\Int(D_\ppp) = \Int(D)_\ppp$.  Then for any $t$-linked extension $A$ of $D$, the following conditions are equivalent.
\begin{enumerate}
\item $A$ is a polynomially complete extension of $D$.
\item $A_\ppp$ is a polynomially complete extension of $D_\ppp$ for every $t$-maximal ideal $\ppp$ of $D$ with finite residue field.
\item $A$ is unramified and has trivial residue field extensions at every $t$-maximal ideal of $D$ with finite residue field.
\item $A$ is a weakly polynomially complete extension of $D$.
\item $A_\ppp$ is a weakly polynomially complete extension of $D_\ppp$ for every $t$-maximal ideal $\ppp$ of $D$ with finite residue field.
\end{enumerate} 
\end{corollary}

A domain $D$ is said to be {\it absolutely polynomially regular} if every extension of $D$ is polynomially regular.  For example, every Dedekind domain is absolutely polynomially regular, while the unique factorization domain $\ZZ[T]$ is not, since the extension $\ZZ[T/2]$ of $\ZZ[T]$ is not polynomially regular by \cite[Example 7.3]{ell}.  Let us say that a domain $D$ is {\it polynomially F-regular} if every flat extension of $D$ is polynomially regular.  For example, every Krull domain is polynomially F-regular, by \cite[Theorem 1.3]{ell}.  In this section we prove several generalizations of this fact, namely, Theorems \ref{polyfreg}, \ref{wpolyfreg}, and \ref{tpolyreg}.

If $D$ is polynomially F-regular, then $D$ is polynomially L-regular, that is, one has $\Int(S^{-1}D) = S^{-1}\Int(D)$ for every multiplicative subset $S$ of $D$.
In fact, as remarked at the end of Section 3 of \cite{ell}, one has the following.

\begin{lemma}\label{easyprop}
Let $D$ be any integral domain.  The domain $D_\ppp$ is absolutely polynomially regular for every prime ideal $\ppp$ of $D$ with infinite residue field, and the following conditions are equivalent.
\begin{enumerate}
\item $D$ is absolutely polynomially regular (resp., polynomially F-regular).
\item $D$ is polynomially L-regular and $D_\ppp$ is absolutely polynomially regular (resp., polynomially F-regular) for every prime ideal $\ppp$ of $D$.
\item $D$ is polynomially L-regular and $D_\ppp$ is absolutely polynomially regular (resp., polynomially F-regular) for every maximal ideal $\ppp$ of $D$ with finite residue field.
\end{enumerate}
\end{lemma}

The following result shows that polynomial F-regularity is a $t$-local condition. 

\begin{theorem}\label{polyfreg}
Let $D$ be an integral domain.  Then $D$ is polynomially F-regular if and only if $D$ is polynomially L-regular and $D_\ppp$ is polynomially F-regular for every $t$-localizing $t$-maximal ideal $\ppp$ of $D$ with finite residue field.
\end{theorem}

\begin{proof}
Necessity of the condition follows from Lemma \ref{easyprop}.  Suppose that $D$ is polynomially L-regular and $D_\ppp$ is polynomially F-regular for every $t$-localizing $t$-maximal ideal $\ppp$ of $D$ with finite residue field, and let $A$ be a flat extension of $D$.  Then $A_\ppp$ is a flat extension of $D_\ppp$ for any $\ppp$, and by Proposition \ref{locallytlinked} one has $A = \bigcap_{\ppp \in \sKr(K/D)} A_\ppp$.  One therefore has $$\Int(D,A) = \bigcap_{\ppp \in \sKr(K/D)} \Int(D_\ppp,A_\ppp) = \bigcap_{\ppp \in \sKr(K/D)} (A\Int(D))_\ppp.$$  Since $A$ is flat as a $D$-module, we have $A \Int(D) = A \otimes_D \Int(D)$, by \cite[Proposition 3.9]{ell}.  Therefore, since $\Int(D)$ is a locally $t$-linked extension of $D$ by Lemma \ref{inttlinked}, it follows from Proposition \ref{tlinkedlemma} that $A \Int(D)$ is a locally $t$-linked extension of $D$.  Therefore by Proposition \ref{locallytlinked} we have $\Int(D,A) =  \bigcap_{\ppp \in \sKr(K/D)} (A \Int(D))_\ppp = A\Int(D)$.  Thus $A$ is a polynomially regular extension of $D$.  This proves that $D$ is polynomially F-regular.  
\end{proof}

Note that the statement obtained from Theorem \ref{polyfreg} by replacing polynomial F-regularity with absolute polynomial regularity does not hold.  For example, the domain $\ZZ[T]$ is a Krull domain, hence polynomially L-regular, and every localization of $\ZZ[T]$ at a $t$-maximal ideal is a DVR and therefore absolutely polynomially regular.  Nevertheless $\ZZ[T]$ is not absolutely polynomially regular.  Thus, polynomially F-regularity is a $t$-local condition, while absolute polynomial regularity is not.

Next, let us say that a domain $D$ is {\it polynomially $t$-regular} if every $t$-linked extension of $D$ is polynomially regular.  We also say that $D$ is {\it polynomially L-$t$-regular} if every locally $t$-linked extension of $D$ is polynomially regular.  Since localizations are flat, flat extensions are locally $t$-linked, and locally $t$-linked extensions are $t$-linked, we clearly have the following implications:

\begin{eqnarray*}
\mbox{absolutely polynomially regular } & \Rightarrow & \mbox{ polynomially }t\mbox{-regular } \\   & \Rightarrow & \mbox{ polynomially L-}t\mbox{-regular } \\ & \Rightarrow & \mbox{ polynomially F-regular } \\ & \Rightarrow & \mbox{ polynomially L-regular.}
\end{eqnarray*}


Like flatness, local $t$-linkedness is a local property.  We can thus derive the following analogue of Lemma \ref{easyprop}.

\begin{lemma}\label{weasyprop}
For any integral domain $D$, the following conditions are equivalent.
\begin{enumerate}
\item $D$ is polynomially L-$t$-regular.
\item $D$ is polynomially L-regular and $D_\ppp$ is polynomially L-$t$-regular for every prime ideal $\ppp$ of $D$.
\item $D$ is polynomially L-regular and $D_\ppp$ is polynomially L-$t$-regular for every maximal ideal $\ppp$ of $D$ with finite residue field.
\end{enumerate}
\end{lemma}

\begin{proof}
Suppose that $D$ is polynomially L-$t$-regular, and let $\ppp$ be a prime ideal of $D$.  Let $A$ be a locally $t$-linked extension of $D_\ppp$.  Then $A$ is a locally $t$-linked
extension of $D$, whence $A$ is a polynomially regular extension of $D$.  It follows, then, from \cite[Lemma 3.5]{ell}, that $A$ is a polynomially regular
extension of $D_\ppp$.  Thus $D_\ppp$ is polynomially L-$t$-regular.  Thus (1) implies (2), and clearly (2) and (3) are equivalent.  Lastly, suppose that (2) holds, and let $A$
be a locally $t$-linked extension of $D$.  Then for any prime $\ppp$ of $D$ the extension $A_\ppp$ of $D_\ppp$ is locally $t$-linked, whence $A_\ppp$ is a polynomially
regular extension of $D_\ppp$ for all $\ppp$.  Thus it follows that $\Int(D,A) = \bigcap_{\ppp \in \Spec(D)} \Int(D_\ppp,A_\ppp) = 
\bigcap_{\ppp \in \Spec(D)} (A\Int(D))_\ppp = A\Int(D)$, whence $A$ is a polynomially regular extension of $D$.  Thus (2) implies (1).
\end{proof}

An analogue of Theorem \ref{polyfreg} for polynomial L-$t$-regularity holds if $\Int(D)$ is assumed flat as a $D$-module.

\begin{theorem}\label{wpolyfreg}
Let $D$ be an integral domain such that $\Int(D)$ is flat as a $D$-module.  Then $D$ is polynomially L-$t$-regular if and only if $D$ is polynomially L-regular and $D_\ppp$ is polynomially L-$t$-regular for every $t$-localizing $t$-maximal ideal $\ppp$ of $D$ with finite residue field.
\end{theorem}

\begin{proof}
Necessity of the condition follows from Lemma \ref{weasyprop}.  Suppose that $D$ is polynomially L-regular and $D_\ppp$ is polynomially L-$t$-regular for every $t$-localizing $t$-maximal ideal $\ppp$ of $D$ with finite residue field, and let $A$ be a locally $t$-linked extension of $D$.  Then $A_\ppp$ is locally $t$-linked extension of $D_\ppp$ for any $\ppp$, and by Proposition \ref{locallytlinked} one has $A = \bigcap_{\ppp \in \sKr(K/D)} A_\ppp$.  One therefore has $$\Int(D,A) = \bigcap_{\ppp \in \sKr(K/D)} \Int(D_\ppp,A_\ppp) = \bigcap_{\ppp \in \sKr(K/D)} (A\Int(D))_\ppp.$$  Since by hypothesis $\Int(D)$ is flat as a $D$-module, we have $A \Int(D) = A \otimes_D \Int(D)$ by \cite[Proposition 3.9]{ell}.  Therefore, since $A$ is a locally $t$-linked extension of $D$, it follows from Proposition \ref{tlinkedlemma} that $A \Int(D)$ is a locally $t$-linked extension of $D$.  Therefore we have $\Int(D,A) =  \bigcap_{\ppp \in \sKr(K/D)} (A \Int(D))_\ppp = A\Int(D)$.  Thus $A$ is a polynomially regular extension of $D$.  This proves that $D$ is polynomially L-$t$-regular.  
\end{proof}

Unlike flatness and local $t$-linkedness, $t$-linkedness is not a local property.  Thus we may prove only a partial analogue of Lemma \ref{easyprop} and Theorem \ref{polyfreg} for the polynomially $t$-regular domains.

\begin{theorem}\label{tpolyreg}
Let $D$ be an integral domain.  Then the first condition below implies the latter two conditions.
\begin{enumerate}
\item[(a)] $D$ is polynomially $t$-regular.
\item[(b)] $D$ is polynomially L-regular and $D_\ppp$ is polynomially $t$-regular for every prime ideal $\ppp$ of $D$.
\item[(c)] $D$ is polynomially L-regular and $D_\ppp$ is absolutely polynomially regular for every $t$-maximal ideal $\ppp$ of $D$.
\end{enumerate}
Moreover, if $\Int(D)$ is flat as a $D$-module, then conditions (a) and (c) are equivalent.
\end{theorem}

\begin{proof}
Suppose that (a) holds.  Let $\ppp$ be a prime ideal of $D$, and let $A$ be a $t$-linked extension of $D_\ppp$.  Now, $D_\ppp$ is a flat
and therefore $t$-linked extension of $D$.  Thus, it follows from Proposition \ref{tlinkedlemma}(a) that $A$ is a $t$-linked extension of $D$.  Thus $A$ is a polynomially regular extension of $D$, since $D$ is polynomially $t$-regular.  Therefore, since $D_\ppp$ is a flat and polynomially regular overring of $D$ and $A$ is a polynomially regular extension of $D$, it follows from \cite[Lemma 3.5]{ell} that $A$ is a polynomially regular extension of $D_\ppp$.  Thus $D_\ppp$ is polynomially $t$-regular, and (a) implies (b).  Suppose, on the other hand, that $\ppp$ is a $t$-maximal ideal of $D$, and let $A$ be any extension of $D_\ppp$.  Then $A \subset \bigcap_{\qqq \in \TM(D)} A_\qqq \subset A_\ppp = A$, whence $A$ is a $t$-linked extension of $D$.  Therefore $A$ is a polynomially regular extension of $D$.  Again, it follows that $A$ is
a polynomially regular extension of $D_\ppp$.  Thus $D_\ppp$ is absolutely polynomially regular.  Thus (a) implies (c).

Suppose now that $\Int(D)$ is flat as a $D$-module, suppose that (c) holds, and let $A$ be a $t$-linked extension of $D$. Then we have $$\Int(D,A) = \bigcap_{\ppp \in \TM(D)} \Int(D_\ppp,A_\ppp) = \bigcap_{\ppp \in \TM(D)} (A\Int(D))_\ppp.$$  Since $\Int(D)$ is flat as a $D$-module, we have $A \Int(D) = A \otimes_D \Int(D)$, and therefore $A\Int(D)$ is a $t$-linked extension of $D$ by Proposition \ref{tlinkedlemma}.  Thus we have
$\Int(D,A) =  \bigcap_{\ppp \in \TM(D)} (A \Int(D))_\ppp = A\Int(D)$.  Thus $A$ is a polynomially regular extension of $D$.  Thus (c) implies (a) and the two conditions are equivalent in this case.
\end{proof}

Finally, we note that Theorem \ref{introthm} of the introduction follows from Theorem \ref{tpolyreg}, Corollary \ref{pvmdext}, and Lemmas \ref{TVPVMD} and \ref{locallyfreecor}.  In fact, a similar argument yields the following result, which has a slightly weaker hypothesis and a slightly weaker conclusion than Theorem \ref{introthm}.

\begin{theorem}\label{introthm2}
Let $D$ be an infinite domain such that $\ppp D_\ppp$ is principal and $\Int(D_\ppp) = \Int(D)_\ppp$ for every $t$-localizing $t$-maximal ideal $\ppp$ of $D$ with finite residue field.  Then we have the following.
\begin{enumerate}
\item[(a)] $\Int(D)$ is locally free as a $D$-module.
\item[(b)] $\Int(S^{-1}D) = S^{-1}\Int(D)$ for every multiplicative subset $S$ of $D$.
\item[(c)] Every locally $t$-linked extension of $D$ is polynomially regular.
\item[(d)] A locally $t$-linked extension of $D$ is polynomially complete if and only if it is unramified and has trivial residue field extensions at every $t$-localizing $t$-maximal ideal of $D$ with finite residue field.
\end{enumerate}
\end{theorem}

\subsection{Stability of regularity conditions}

Recall that an extension $A \supset D$ of domains is said to be {\it polynomially complete} if $\Int(D,A) = \Int(A)$, and the extension is said to be a {\it weakly polynomially complete} if $\Int(D) \subset \Int(A)$.  As noted at the end of Section 3 of \cite{ell}, the absolutely polynomially regular domains are stable under (weakly) polynomially complete extensions.  Thus, for example, if $D$ is absolutely polynomially regular, then $\Int(D^\XX)$ is absolutely polynomially regular for any set $\XX$.  The following result shows that all of the regularity conditions we have studied are stable under the appropriate polynomially complete extensions.

\begin{lemma}\label{extext}
If $A$ is a weakly polynomially complete extension of a domain $D$ and $B$ is a polynomially regular extension of $D$ containing $A$, then $B$ is a polynomially regular extension of $A$.  Moreover, we have the following.
\begin{enumerate}
\item[(a)] $D$ is absolutely polynomially regular if and only if every (weakly) polynomially complete extension of $D$ is absolutely polynomially regular.
\item[(b)] $D$ is polynomially $t$-regular if and only if every (weakly) polynomially complete $t$-linked extension of $D$ is polynomially $t$-regular.
\item[(c)] $D$ is polynomially L-$t$-regular if and only if every (weakly) polynomially complete locally $t$-linked extension of $D$ is polynomially L-$t$-regular.
\item[(d)] $D$ is polynomially F-regular if and only if every (weakly) polynomially complete flat extension of $D$ is polynomially F-regular.
\item[(e)] Every flat overring of $D$ is polynomially complete, and $D$ is polynomially L-regular if and only if every flat overring of $D$ is polynomially L-regular.
\end{enumerate}
\end{lemma}

\begin{proof}
The hypotheses on $A$ and $B$ mean that $D \subset A \subset B$ and $\Int(D) \subset \Int(A)$, and $\Int(D,B) = B \Int(D)$.  Thus we have
$$\Int(A,B) \subset \Int(D,B) = B \Int(D) \subset B \Int(A) \subset \Int(A,B),$$ whence $\Int(A,B) = B\Int(A)$.  Thus $B$ is a polynomially regular extension of $A$.  Statement (a) follows immediately.  Statement (b) then follows from Proposition \ref{tlinkedlemma}(a), and statements (c) and (d) follows similarly.  Finally, by \cite[Lemma 3.5]{ell}, every flat overring of $D$ is polynomially complete, and thus (e) follows, since a flat overring of a flat overring is a flat overring.
\end{proof}

\subsection{Application to tensor product decompositions}

Recall that if $D$ is a domain such that the canonical $D$-algebra homomorphism
$\theta_\XX: \bigotimes_{X \in \XX} \Int(D) \longrightarrow \Int(D^\XX)$, where the tensor product is over $D$, is an isomorphism (resp., surjective) for every set $\XX$, then we say that $D$ is {\it polynomially composite} (resp., {\it almost polynomially composite}).  
As of yet we do not know an example of a domain that is not polynomially composite.

\begin{proposition}\label{localcomposite}
Let $D$ be any integral domain with quotient field $K$.  The domain $D_\ppp$ is polynomially composite for every strong polynomial prime $\ppp$ of $D$, and the following conditions are equivalent.
\begin{enumerate}
\item $D_\ppp$ is polynomially composite (resp., almost polynomially composite) for every prime ideal $\ppp$ of $D$.
\item $D_\ppp$ is polynomially composite (resp., almost polynomially composite) for every $t$-localizing $t$-maximal ideal $\ppp$ of $D$ with finite residue field.
\item $D_\ppp$ is polynomially composite (resp., almost polynomially composite) for every prime ideal $\ppp$ of $D$ such that $\ppp D_\ppp$ is a weakly associated prime of $K/D_\ppp$ with finite residue field.
\item $D_\ppp$ is polynomially composite (resp., almost polynomially composite) for every weak int prime $\ppp$ of $D$.
\end{enumerate}
Moreover, if any of the above conditions holds and $D$ is polynomially L-regular, then $D$ is polynomially composite (resp., almost polynomially composite).
\end{proposition}

\begin{proof}
If $\ppp$ is a strong polynomial prime of $D$, that is, if $\Int(D_\ppp) = D_\ppp[X]$, then $\Int(D_\ppp)$ is free over $D_\ppp$, and thus by \cite[Proposition 6.8]{ell} the domain $D_\ppp$ is polynomially composite.  The four conditions are therefore equivalent by Lemma \ref{implications}.  Finally, if condition (1) holds and $D$ is polynomially L-regular, then $D$ is polynomially composite (resp., almost polynomially composite) by \cite[Corollary 6.6]{ell}.
\end{proof}

By the following result, polynomial $t$-regularity, polynomial L-$t$-regularity, polynomial F-regularity, and polynomial L-regularity are all useful conditions for studying (almost) polynomial compositeness.

\begin{theorem}\label{comp}
Let $D$ be an integral domain.  We have the following.
\begin{enumerate}
\item[(a)] If $D$ is polynomially $t$-regular or polynomially L-$t$-regular, then $D$ is almost polynomially composite.
\item[(b)] If $D$ is polynomially F-regular and $\Int(D)$ is flat as a $D$-module, then $D$ is polynomially composite.
\item[(c)] If $D$ is polynomially L-regular and $\Int(D)$ is locally free as a $D$-module, then $D$ is polynomially composite.
\end{enumerate}
\end{theorem}

\begin{proof}
We note first that (b) holds by \cite[Proposition 6.8]{ell} and (c) holds by \cite[Proposition 6.10]{ell}.  To prove (a), suppose that $D$ is polynomially L-$t$-regular.  It follows that $D$ is not a finite field.  By Lemma \ref{inttlinked}, the extension $\Int(D)$ of $D$ is locally $t$-linked, and since $\Int(D)$ is a polynomially complete extension of $D$, it follows from Lemma \ref{extext}(c) that $\Int(D)$ is also polynomially L-$t$-regular.  In particular, $\Int(D)$ is polynomially L-regular, and therefore $\Int(D^2) = \Int(\Int(D))$ is a locally $t$-linked extension of $\Int(D)$ by Lemma \ref{inttlinked}(b).  It follows from Proposition \ref{tlinkedlemma} that $\Int(D^2)$ is a locally $t$-linked extension of $D$.  By induction it follows that $\Int(D^n)$ is a locally $t$-linked, hence polynomially regular, extension of $D$ for every positive integer $n$.  Therefore, by \cite[Proposition 6.3]{ell}, $D$ is almost polynomially composite. 
\end{proof}

\end{document}